\documentclass[11pt]{article}
\usepackage{amsmath,amssymb,amsthm,amsfonts,amstext,amsbsy,amscd,color}
\usepackage{cases}
\usepackage{cancel}
\usepackage[dvipsnames]{xcolor}
\usepackage{graphicx}
\usepackage[utf8]{inputenc}
\usepackage{mathtools}
\usepackage{marvosym}
\usepackage{manfnt}
\usepackage{bm}
\usepackage[labelfont=bf]{caption}

\usepackage[
  backend=biber,   
  style=numeric,   
  giveninits=true  
]{biblatex}

\usepackage{bbm}
\usepackage{enumitem}
\usepackage{amsmath}
\usepackage{amssymb}
\usepackage{hyperref}
\usepackage[notcite, notref]{showkeys}
\usepackage[left=1cm,right=2.5cm,top=2cm,bottom=2.5cm]{geometry}
\usepackage{subcaption}
\definecolor{myorange}{rgb}{0.9568,0.4941,0.1961}

\def\<{\langle}
\def\>{\rangle}

\def\Chi{\raise .3ex
\hbox{\large $\chi$}}

\newcommand{\dd}{\text{d}}
\newcommand{\dst}{\displaystyle}
\newcommand{\be}{\begin{equation}}

\newcommand{\f}{\frac}
\newcommand{\ee}{\end{equation}}
\newcommand{\bea}{$$ \begin{array}{lll}}
\newcommand{\eea}{\end{array} $$}
\newcommand{\bi}{\begin{itemize}}
\newcommand{\ei}{\end{itemize}}

\newcommand\blfootnote[1]{%
	\begingroup
	\renewcommand\thefootnote{}\footnote{#1}%
	\addtocounter{footnote}{-1}%
	\endgroup
}
\numberwithin{equation}{section}
\newtheorem{satz}{Satz}[section]
\newtheorem{theo}[satz]{Theorem}

\newtheorem{lemma}{Lemma}

\newtheorem{definition}{Definition}

\newtheorem{remark}{Remark}

\DeclareMathOperator{\R}{{\mathbb R}}

\providecommand{\eps}{\varepsilon}

\renewenvironment{proof}{\noindent{\bf Proof.}}{\hfill
  $\blacksquare$\par\noindent}

\begin{document}
%\dateposted{demain}
%\thanks{}
\title{Steady States of Transport-Coagulation-Nucleation Models}

\author{J. Delacour\footnote{ENS de Lyon site Monod,
UMPA UMR 5669 CNRS
46, allée d’Italie
69364 Lyon Cedex 07, France.} \and M. Doumic{{$^\dag$}}\footnote{Sorbonne Universités, Inria, UPMC Univ Paris 06, Lab. J.L. Lions  UMR CNRS 7598, Paris, France.}\footnote{École polytechnique, CMAP, Rte de Saclay, 91120 Palaiseau, France.} \and C. Moschella\footnote{
Mathematical Institute, University of Oxford,
Andrew Wiles Building, Radcliffe Observatory Quarter,
Woodstock Road, Oxford OX2 6GG, UK.
} \footnote{University of Vienna, Faculty of Mathematics, Oskar-Morgenstern-Platz 1, 1090 Vienna, Austria.} \and  C. Schmeiser{{$^\P$}}}

\maketitle

	\vspace{-11pt}
	\begin{abstract}
		\noindent\small
To model the dynamics of polymers formed through nucleation, elongated  by polymerisation, shortened by depolymerisation and subject to aggregation reactions, we study a nonlinear integro-differential equation. Growth and shrinkage are described by transport terms,  nucleation by a positive boundary condition, and aggregation by a Smoluchowski coagulation kernel. Our main result is the existence of steady states for the multiplicative coagulation kernel despite this   kernel producing gelation in finite time for the pure coagulation equation. This is made possible by a sufficiently strong decay rate for large polymers.
Beyond the existence result, the qualitative properties of the steady states are illustrated through explicit examples and  numerical experiments. The analytical results connect the
growth behaviour of the transport velocity and of the coagulation kernel to the decay properties of steady states.
	\end{abstract}

%\tableofcontents

\section{Introduction} 
Coagulation equations constitute a fundamental class of integro-differential equations that describe the evolution of size-distributed particle systems undergoing binary aggregation. These models have been widely applied across the physical sciences, particularly in the context of aerosol dynamics \cite{drake1972general}, where the interaction and aggregation of gas particles or droplets are governed by collision mechanisms encoded in a coagulation kernel \cite{friedlander2000smoke, seinfeld2016atmospheric}. In the biological sciences, coagulation-type equations have been employed to model processes such as the polymerisation of macromolecules \cite{ziff1980kinetics}.
Given these numerous applications, coagulation equations -- and more generally, coagulation-fragmentation equations (see, e.g., \cite{banasiak2019analytic, book} for comprehensive references on the topic) -- have been extensively studied from both analytical and numerical perspectives. In contrast, coagulation equations incorporating a transport term in the size variable have received significantly less attention, despite their relevance in modelling systems where particles grow or shrink continuously in addition to coalescence. Such transport term, analogous to those used in structured population models such as the McKendrick-von Foerster equation \cite{ackleh1997modeling, perthame2006transport}, may arise naturally in contexts where aggregates self-grow by surface accretion or shrink through degradation. An example and motivation for this work is the formation of protein aggregates in cells
as a preparation for their recycling by autophagy \cite{delacour2022mathematical}.

We are concerned with the existence of steady state solutions of the following coagulation-transport-nucleation model:
\large
\begin{eqnarray}
%\begin{array}{ll}
\label{eq:main}
\partial_t f (t,x) + \partial_x\big(v(x)f(t,x)\big) &=&\f{1}{2} \int\limits_0^x K(x-y,y) f(t,x-y)f(t,y)\dd y \\
&&-\int\limits_0^\infty K(x,y) f(t,y) f(t,x) \dd y\,, \nonumber
\\
f(t,0)&=&f_0\,,
%\end{array}
\label{eq:main2}
\end{eqnarray}
\normalsize
where \(f(t,x)\) is the particle size distribution function at time \(t\ge0\), with particle size \(x\in[0,\infty)\).

The function $v(x) \in C^1(0,\infty)$ 
represents the size-dependent transport velocity (i.e. the rate of  growth or shrinkage of the particles by addition or depletion of monomers).  Contrarily to the case of the Lifshitz-Slyozov equation~\cite{laurencot2001lifshitz,laurencot2001weak}, the concentration of monomers is assumed to be constant so that the transport rate only depends on the size of the clusters.

The boundary condition $f(t,0)=f_0>0$ models nucleation, i.e. the continuous injection of monomeric units into the system, as observed in biological contexts such as protein 
synthesis, cell renewal, or polymer production~\cite{calvo2018long,calvo2024long,prigent2012efficient}. It only makes sense if small polymers are not decomposed immediately, i.e. if $v(x)>0$ close to $x=0$. 

We define $K(x,y)$ as a symmetric coagulation kernel describing the coagulation rate between particles of size $x$ and $y$. We assume that $K$ satisfies
\begin{equation*}
\forall \, x\geq 0,\, y\geq 0, \qquad K(x,y)=K(y,x),  
\end{equation*}
which formally implies, as long as all quantities are integrable:
\begin{equation}\label{eq:balance}
\left\{ \begin{array}{ll}
\f{\dd}{\dd t} M_0 \coloneqq \f{\dd}{\dd t} \dst\int_0^\infty f(t,x) \dd x = v(0)f_0-\f{1}{2}\int_0^\infty\int_0^\infty K(x,y) f(t,x)f(t,y) \dd x\dd y,
\\ \\
\f{\dd}{\dd t} M_1 \coloneqq \f{\dd}{\dd t} \dst \int_0^\infty x f(t,x) \dd x = \int_0^\infty v(x) f(t,x) \dd x,
\end{array}\right.
\end{equation}
where in the first equation we have assumed that the flux at infinity vanishes, i.e.,
\be\label{BC-inf-0}
\lim_{x\to\infty} v(x) f(t,x) = 0  \,.
\ee
The first equation shows the total balance of the number of individuals, increased by nucleation and decreased by coagulation, whereas the second shows that the change in the total mass is only due to the transport term. These balance equations remain true as long as no particle is lost by gelation, i.e. the formation of infinite-size particles by coagulation~\cite{hendriks1983coagulation}, which can also happen instantaneously~\cite{van1987possible,cristian2025note}.

The problem of the existence of steady-state solutions has been extensively tackled in the context of coagulation-fragmentation equations. When both coagulation and fragmentation processes are present, the dynamics are reversible, and detailed balance conditions can often be written explicitly (see, e.g., Section 2.3.3 of \cite{banasiak2019analytic}). In contrast, when only coagulation or only fragmentation occurs, one might intuitively expect that no non-trivial steady states can exist, due to the inherently irreversible nature of these dynamics. Surprisingly, this intuition fails in certain cases: exact steady-state solutions have been constructed for specific classes of coagulation and multiple-fragmentation kernels, as shown in \cite{dubovskii1992exact}. In a related direction, the existence and characterization of steady states  has also been investigated for transport-fragmentation models \cite{doumic2010eigenelements}, where the eigenelements of the associated eigenproblem are used to analyze the long-time behaviour of the system~\cite{perthame2006transport}. Furthermore, the existence of stationary non-equilibrium solutions has been established in \cite{ferreira2021stationary} for coagulation equations -- both in the discrete and continuous setting -- including a source term representing the continuous injection of small-sized particles. This scenario shares structural similarities with our setting, where the boundary condition at zero models the persistent input of monomeric units.

 The growth-coagulation equation~\eqref{eq:main} was first introduced by Lifshitz and Slyozov~\cite{LifshitzSlyozov_1961}, with a nonlinear  coupling of the equation with the monomers dynamics through the growth term, and no nucleation since $v(0)$ was assumed nonpositive. This model is also called the "Lifshitz-Slyozov equation with encounters". Studies have focused on well-posedness~\cite{collet1999lifshitz,laurencot2001lifshitz}, self-similar solutions~\cite{herrmanna2009self}, and numerical solution to investigate by simulations the large time asymptotics~\cite{goudon2013simulations}. Other articles have focused on the well-posedness for the growth-coagulation equation with an everywhere nonnegative growth rate $v$, either with a nonlocal nucleation term coming from larger aggregates like in the renewal equation, for instance to model phytoplankton aggregation~\cite{ackleh1997modeling,ackleh2006existence}, or with no nucleation~\cite{giri2025well}.

The aim of this work is to identify conditions on the transport velocity, under which \eqref{eq:main} admits non-trivial steady-state profiles, especially in the case of the multiplicative coagulation 
kernel $K(x,y)=xy$ for which it has been proven that gelation occurs in finite time~\cite{EscobedoMischlerPerthame_2002}. We shall see in the following section that this is not the case for an everywhere positive transport velocity, since this leads to gelation in finite time. Therefore
it will be assumed that large polymers tend to shrink, i.e. the transport velocity $v(x)$ is negative for large $x$.  In Section \ref{sec:inverse_problem} the inverse problem of prescribing
a steady state solution and determining corresponding transport velocities will be considered. Several explicit solutions will serve as illustrations. Section \ref{sec:existence_steady} 
contains an existence result for steady states. The main assumption is that the transport velocity is not only negative for large $x$, but tends to $-\infty$ fast enough as $x\to\infty$.
Some qualitative properties of steady states are discussed formally, in particular the fact that there can be a singularity at the zero of the transport velocity. Finally, our results are
illustrated by numerical experiments in Section \ref{sec:num}.

\section{The choice of the transport velocity}\label{sec:gelation}
We start this section by deriving explicit information on the solution of \eqref{eq:main} for the constant positive transport velocity
 \( v \equiv 1 \). In this case both transport and coagulation increase the polymer size, and therefore convergence to a steady state cannot be expected. 
 However the solution behaviour also depends on the coagulation rate. 
 
\paragraph{Constant transport velocity and constant coagulation rate:} 
For $v\equiv K\equiv 1$ we can derive from~\eqref{eq:balance} the closed system of ordinary differential equations (ODEs) 
\[
   \dot M_0 = f_0 - \frac{1}{2}M_0^2 \,,\qquad \dot M_1 = M_0 \,,
\]
governing the time evolution of the moments 
\[
M_0(t) = \int_0^\infty f(t,x)\dd x, \quad M_1(t) = \int_0^\infty x f(t,x)\dd x \,.
\]
Note that the total particle number $M_0$ is influenced by nucleation and coagulation, and the total mass $M_1$ only by the transport.
The total particle number tends to the limit $\sqrt{2f_0}$ as $t\to\infty$, and the total mass grows at an asymptotically linear rate.

\paragraph{Constant transport velocity and multiplicative coagulation rate:} 
For $v\equiv 1$ and $K(x,y)=xy$, the equations for the moments are again a closed system:
\[
   \dot M_0 = f_0 - \frac{1}{2}M_1^2 \,,\qquad \dot M_1 = M_0 \,,
\]
and differentiating the equation for \( M_1 \) yields the second-order Hamiltonian ODE
\begin{equation}\label{eq:second_order_derivative_moment}
\ddot M_1 = f_0 - \frac{1}{2} M_1^2 \,,
\end{equation}
with the first integral
\[
  H = \frac{1}{2} \left( \dot M_1\right)^2 - f_0 M_1 + \frac{1}{6} M_1^3 \,,
\]
whose value is determined by the initial data. The solution of the first order ODE
\[
\dot M_1 = \sqrt{2H + 2 f_0 M_1 - \frac{1}{3} M_1^3}
\]
increases and reaches the zero of the right hand side in finite time $T$ because of the square root. At time $T$ with $M_0(T)=0$ and $M_1(T)>0$ the model breaks down, 
which signals the occurrence of \emph{gelation} in finite time, i.e. all the mass has been moved to $x=\infty$ at $t=T$. 

\paragraph{Choosing a transport velocity with changing sign:}
The observations from above remain true. Equations~\eqref{eq:balance} show that changes of the total particle number are caused by nucleation and coagulation, and changes of the total mass by the transport, 
therefore a steady state needs to balance nucleation and coagulation, but also the production and consumption of monomers by the transport process.
Obviously, the transport velocity needs to change sign so that the right-hand side of the second equation in~\eqref{eq:balance} can vanish.
The simplest possible model assumes the existence of an equilibrium size $x_0>0$ such that
\begin{equation}
\label{eq:velocity_profile}
\begin{aligned}
&v(x) > 0 \,, \quad &&\text{for } 0\le x < x_0 \,, \\
&v(x_0) = 0 \,, \quad &&\text{for } x_0 > 0 \,, \\
&v(x) < 0 \,, \quad &&\text{for } x > x_0 \,.
\end{aligned}
\end{equation}
For well-posedness of the problem, this requires an additional boundary condition at $x=\infty$, such as \eqref{BC-inf-0}.
More generally, we shall pose
\be\label{BC-inf}
\lim_{x \to \infty} v(x) f(x,t) = j_\infty \le 0\,,
\ee
where $j_\infty < 0$ could be interpreted as nucleation of infinite size particles out of a background gel.

\section{The inverse problem -- explicit examples of steady state solutions}\label{sec:inverse_problem}

As a source of illustrative explicit examples we consider the inverse problem of prescribing a steady state $f(x)$ and determining corresponding transport velocities $v(x)$,
such that
\be\label{steadystate}
   \partial_x\big(v(x)f(x)\big) =\f{1}{2} \int\limits_0^x K(x-y,y) f(x-y)f(y)\dd y - \int\limits_0^\infty K(x,y) f(y) f(x) \dd y\,.
\ee
Obviously, $v(x)$ is only unique up to adding multiples of $1/f(x)$. We shall be interested in the solutions with minimal growth of $|v(x)|$ as $x\to\infty$, since we are mainly 
interested in the question, how much growth of $|v(x)|$ is needed for a certain decay behaviour of $f(x)$. Obviously, a constant multiple of $1/f(x)$ produces a
constant contribution to the flux. We shall be looking for $v(x)$ with growth smaller
than $1/f(x)$ which, if it exists, is unique. Note that this corresponds to  $j_\infty=0$ in \eqref{BC-inf}. 
As an equation for $f$, \eqref{steadystate} is singular at zeroes of $v$. Therefore steady states with singularities have to be expected in general, as is shown in the following section. The given smooth steady states considered in this section are  very special and exist only for appropriate combinations of the data 
$f_0$ and $v$.

\paragraph{Exponential steady states:}
For $f(x)= f_0 e^{-\lambda x}$, $\lambda>0$, \eqref{steadystate} becomes
$$
   v' - \lambda v = \f{f_0}{2} \int\limits_0^x K(x-y,y) \dd y - f_0\int\limits_0^\infty K(x,y) e^{-\lambda y} \dd y \,.
$$
If $K$ is polynomial, so is the right hand side, and a polynomial solution $v$ (which is in this case the one we are looking for) can be found by an ansatz. This is how
we proceed in the first three examples below.

\begin{enumerate} 
\item $K(x,y)=1$:\quad $v(x) = \frac{f_0}{2\lambda}\left(\frac{1}{\lambda} - x\right)$\,.
    \item $K(x,y) = x + y$:\quad $v(x) = \frac{f_0}{2\lambda}\left( \frac{\sqrt{2}}{\lambda} + x\right)\left( \frac{\sqrt{2}}{\lambda} - x\right)$\,.
    \item $ K(x,y) = xy$:\quad $v(x) = \frac{f_0}{12\lambda^4}\left(6 + 6\lambda x - 3(\lambda x)^2 - (\lambda x)^3 \right)$\,.\\
    It can be easily checked that $v$ satisfies \eqref{eq:velocity_profile}, as obviously also in the previous 2 examples. In all 3 examples so far the polynomial degree of $v$ is by 
    one higher than the polynomial degree of $K$.
    \item Generalizing 1. and 3., for $K(x,y) = (xy)^\beta$, $\beta\ge 0$, we have to solve
    \begin{equation}\label{eq:v(x)_x^alpha,y^beta}
    v'-\lambda v = \frac{f_0}{2} B(\beta+1,\beta+1)x^{1+2\beta} - \frac{f_0}{\lambda^{\beta+1}}\Gamma(\beta+1)x^\beta\,,
\end{equation}
with the Beta-function $B$ and the Gamma-function $\Gamma$. At least formally, it is straightforward to construct an asymptotic expansion for a solution:
$$
   v(x) = - \frac{f_0}{2\lambda} B(\beta+1,\beta+1)x^{1+2\beta} + O(x^{2\beta}) \,,\qquad \text{as } x\to\infty \,.
$$
Note that the above mentioned observation deg$(v)=$ deg$(K)+1$ still holds.
\end{enumerate}

\paragraph{Steady states with algebraic decay:}
We consider \eqref{steadystate} with $f(x) = f_0 (1+x^\alpha)^{-1}$ with $\alpha>1$ for integrability.
\begin{enumerate}
    \item
    For $K(x,y)=(xy)^\beta$, $\beta\ge 0$, $\alpha>\beta+1$, we obtain
 \be\label{steady-alg}
    \frac{1}{f_0}\partial_x\left(\frac{v(x)}{1+x^\alpha}\right)= \int_0^{x/2} \frac{(x-y)^\beta}{1+(x-y)^\alpha} \frac{y^\beta}{1+y^\alpha}\,\dd y - \frac{x^\beta}{1+x^\alpha}\int_0^\infty \frac{y^\beta}{1+y^\alpha}\,\dd y \,.
\ee
We shall investigate the asymptotic behaviour of the right hand side as $x\to\infty$. This is subtle, since both terms are asymptotically equivalent to 
$$
   x^{\beta-\alpha} \int_0^\infty \frac{y^\beta}{1+y^\alpha}\,\dd y \,.
$$
The deviation of the first term from this is given by 
\begin{align*}
  & \int_0^{x/2} \frac{(x-y)^\beta}{1+(x-y)^\alpha} \frac{y^\beta}{1+y^\alpha}\,\dd y - x^{\beta-\alpha} \int_0^\infty \frac{y^\beta}{1+y^\alpha}\,\dd y \\
  &= x^{\beta-\alpha} \int_0^{x/2} \left( \frac{(1-y/x)^\beta}{x^{-\alpha}+(1-y/x)^\alpha} - 1\right) \frac{y^\beta}{1+y^\alpha}\,\dd y - x^{\beta-\alpha} \int_{x/2}^\infty \frac{y^\beta}{1+y^\alpha}\,\dd y    
  =: I_1 - I_2 \,.
\end{align*}
For the first term on the right hand side there are 2 cases: For $\alpha>\beta+2$ we have
$$
  I_1 \approx (\alpha-\beta)x^{\beta-\alpha-1} \int_0^\infty \frac{y^{\beta+1}}{1+y^\alpha}\,\dd y \,.
$$
The case $\beta+1<\alpha<\beta+2$ is a little more involved:
$$
  I_1 \approx x^{2\beta-2\alpha+1} \int_0^{1/2} \left( (1-z)^{\beta-\alpha} - 1\right)z^{\beta-\alpha}\dd z \,.
$$
Note that the existence of the integral relies on the assumption $\alpha<\beta+2$. 

We also have
$$
  I_2 \approx \frac{2^{\alpha-\beta-1}}{\alpha-\beta-1} x^{2\beta-2\alpha+1} \,.
$$
In the first case above, this is small compared to $I_1$; in the second case it is of the same order of magnitude.

For the last term in \eqref{steady-alg} we obtain
$$
  \frac{x^\beta}{1+x^\alpha}\int_0^\infty \frac{y^\beta}{1+y^\alpha}\,\dd y -
  x^{\beta-\alpha} \int_0^\infty \frac{y^\beta}{1+y^\alpha}\,\dd y = O(x^{\beta-2\alpha}) \,,
$$
which, in both cases, is small compared to $I_1$. Now we use these results in \eqref{steady-alg}, which we integrate from $x$ to $\infty$ with the assumption $j_\infty=0$:
$$
   v(x) \approx \left\{ \begin{array}{ll} 
   -f_0 c_{\alpha,\beta}\, x^\beta \,,& \text{for }\alpha-\beta>2 \,,\\
   -f_0 \tilde c_{\alpha,\beta}\, x^{2\beta-\alpha+2} \,,& \text{for } 1<\alpha-\beta<2 \,,
   \end{array}\right.
$$
with
$$
  c_{\alpha,\beta} = \int_0^\infty \frac{y^{\beta+1}}{1+y^\alpha}\,\dd y \,,\qquad
  \tilde c_{\alpha,\beta} = \frac{1}{2(\alpha-\beta-1)} \int_0^{1/2} \left(1- (1-z)^{\beta-\alpha}\right)z^{\beta-\alpha} \dd z + \frac{2^{\alpha-\beta-2}}{(\alpha-\beta-1)^2}\,.
$$

\item For $K(x,y)=x+y$ and now with the assumption $\alpha>2$, analogous arguments lead to
$$
   v(x) \approx -f_0 c_\alpha x \,,\qquad\text{with } 
   c_\alpha = \frac{\alpha}{\alpha-1} \int_0^\infty \frac{y}{1+y^\alpha}\,\dd y >0\,.
$$
\end{enumerate}

We conclude this section by collecting our results for the multiplicative kernels
$K(x,y)=(xy)^\beta$, $\beta\ge 0$, from the point of view of the forward problem
of finding $f$ for given $v$. It turns out that the picture is rather counter-intuitive and confusing, since the intuition that stronger growth of $v$ leads
to faster decay of $f$ is only partially true. 

Let the velocity satisfy 
$$
  v(x) \approx -f_0 c\, x^\gamma \,,\qquad\text{as } x\to\infty \,.
$$
If $\gamma=1+2\beta$, we expect exponential decay of $f$, whereas when $\beta\le\gamma < \beta+1$, we expect algebraic decay. This is a partial result, where the intuition is true. However, in the exponential case, from the above computations we expect a rate $\lambda = \frac{B(\beta+1,\beta+1)}{2c}$ and, thus,
stronger decay for a smaller constant in $v$.

For $\gamma=\beta$ we expect algebraic decay of $f$ like $x^{-\alpha}$ with $\alpha\in (\beta+2,\infty)$ determined from $c_{\alpha,\beta}=c$. Here again the intuition fails, since $c_{\alpha,\beta}$ is decreasing in $\alpha$ at least close
to the ends of the $\alpha$-interval.

Finally the last counter-intuitive result: For $\beta<\gamma<\beta+1$, we expect 
algebraic decay with a rate $\alpha = 2\beta+2-\gamma$. And very finally an open question: What if, in the last case, $c\ne \tilde c_{\alpha,\beta}$?

\section{Existence and properties of steady states \label{sec:existence_steady}}
%\subsection{Case of the multiplicative kernel \texorpdfstring{$\boldsymbol{{K(x,y)=xy}}$}{TEXT}}\label{sec:existence_multiplicative}

The stationary transport coagulation problem with multiplicative kernel reads
\begin{equation}
\begin{array}{ll}
&\partial_x (v(x)f(x)) = \frac{1}{2} \dst\int_0^x (x-y)f(x-y)y f(y)\dd y - \int_0^\infty x f(x) y f(y) \dd y \,, \label{eq_1s} 
\\ 
& f(0)=f_0 \,.
\end{array}
\end{equation}
As explained in Section \ref{sec:gelation}, to ensure the existence of non-trivial steady-state solutions, the function $v(x)$ is assumed to exhibit a single transversal sign change,  namely 
\begin{equation}\label{eq_3s}
v(x) = (x_0-x)w(x)\,, \qquad \, w \in C^1([0,+\infty))\,, \qquad w(x)>0 \quad\text{for } x\geq0\,.
\end{equation}
We shall look for a solution, where the transport flux $j= vf$ tends to a nonnegative value as $x\to\infty$. Therefore we complete the problem formulation by 
\be\label{BC-infty}
  \lim_{x\to\infty} v(x)f(x) = j_\infty \le 0 \,.
\ee
At first glance it looks wrong to have two boundary conditions for a first order ODE. However, due to the sign change of the velocity, this is a singular ODE,
where the problems for $x\in [0,x_0)$ and for $x\in (x_0,\infty)$ can be treated separately. Actually the problem on $[0,x_0)$ is closed and its solution can be seen as  
data for the problem on $(x_0,\infty)$.

With \eqref{BC-infty}, formal integration of \eqref{eq_1s} leads to the equation
$$
  j_\infty -v(0)f_0 = -\frac{1}{2}M_1[f]^2\,,
$$
for the first order moment
$$
   M_1[f] := \int_0^\infty xf(x)\dd x \,.
$$
Therefore the value $\mathcal{M}_1 := \sqrt{2(v(0)f_0 - j_\infty)}$ will be fixed from now on. For existence of a non-trivial stationary solution, the transport needs to be strong enough to
counteract the coagulation. This will be guaranteed by an additional assumption:
\begin{equation}\label{v-ass}
  \exists A>0,\, \overline x > x_0,\, \beta >2:\quad v(x) \ge -Ax^\beta \quad\text{for } x\ge \overline x \,.
\end{equation}
We now define the notion of weak stationary solutions to the transport  coagulation equation \eqref{eq_1s} to be used in the sequel.

\begin{definition}\label{def:weak_sol} A weak solution of \eqref{eq_1s}, \eqref{BC-infty} is a function $f\ge 0$ such that $xf\in L^1(\R^+)$ and
\begin{eqnarray*}
   && \Phi(\infty)j_\infty - \Phi(0)v(0)f_0 -\int_0^\infty \Phi'(x) v(x)f(x)\dd x  \\
   && = \frac{1}{2} \int_0^\infty \int_0^\infty xf(x)y f(y)\Phi(x+y)\dd x\,\dd y -  \mathcal{M}_1\int_0^\infty xf(x)\Phi(x) \dd x \,,       
\end{eqnarray*}
for all $\Phi \in C^1_B([0, \infty))$, where $\Phi'$ has compact support.
\end{definition}

The difficulties we face are connected to the singularity at $x=x_0$ and
to the behaviour as $x\to\infty$. Therefore we first consider a regularised version
of \eqref{eq_1s}, \eqref{BC-infty} with the transport flux as the unknown:
\begin{equation}
\begin{array}{ll}
&\partial_x j_\eps + \frac{x\mathcal{M}_1}{v_\eps}j_\eps = h\left[j_\eps/v_\eps\right] \,, \label{eq_1-reg} \\ 
& j_\eps(0)=v(0)f_0 \,,\qquad j_\eps(1/\eps) = j_\infty\,,
\end{array}
\end{equation}
with $0<\eps<1/x_0$, with
$$
v_\eps(x) = \left\{ \begin{array}{ll} v(x)+\eps &\text{for } x\le x_0 \,,\\
             v(x)-\eps &\text{for } x>x_0 \,,\end{array}\right.
$$
implying $|v_\eps|\ge \eps$, and with
$$
  h[f](x) := \frac{1}{2} \dst\int_0^x (x-y)f(x-y) y f(y)\dd y \,.
$$
This defines $j_\eps$ on $[0,1/\eps]$, and it is 
extended to $[0,\infty)$ by the definition $j_\eps=0$ on $(1/\eps,\infty)$.
The regularised problem can be written as a fixed point problem:
\begin{equation}\label{eq:g_formula_reg}
  j_\eps(x) = F_\eps[j_\eps](x) := \left\{ \begin{array}{ll} 
  G_\eps(x,0)v(0)f_0 + \int_0^x G_\eps(x,y)h\left[j_\eps/v_\eps\right](y)\dd y \quad & \text{for }0\le x\le x_0\,,\\
  G_\eps(x,1/\eps)j_\infty -\int_x^{1/\eps} G_\eps(x,y) h\left[j_\eps/v_\eps\right](y)\dd y & \text{for }x_0<x\le 1/\eps\,,\\
  0 & \text{for } x> 1/\eps \,,
     \end{array}\right.
\end{equation}
with 
$$
  G_\eps(x,y) = \exp\left(-\mathcal{M}_1 \int_y^x \frac{z}{v_\eps(z)}\dd z\right) \,.
$$
Formal integration of \eqref{eq_1-reg} gives
$$
  j_\infty - v(0)f_0 + \mathcal{M}_1\,M_1[F_\eps[j]/v_\eps] = \frac{1}{2} M_1[j/v_\eps]^2  \,.
$$
This shows that the correct value of the first order moment is preserved by the fixed point map:
$$
  M_1[j/v_\eps]=\mathcal{M}_1 \quad\Longrightarrow\quad M_1[F_\eps[j]/v_\eps]=\mathcal{M}_1 \,.
$$

\begin{lemma}\label{lem:ex-reg}
Let assumption \eqref{eq_3s} and $0<\eps < 1/x_0$ hold. Then problem \eqref{eq_1-reg} has a solution $j_\eps\in C^1([0,x_0])\times C^1([x_0,1/\eps])$ with $j_\eps/v_\eps\ge 0$, $M_1[j_\eps/v_\eps]=\mathcal{M}_1$, and
\be\label{jeps-est}
   |j_\eps(x)| \le J := \max\{v(0)f_0,-j_\infty\} + v(0)f_0 - j_\infty \,,\qquad x\ge 0\,.   
\ee
\end{lemma}

\begin{remark}
As mentioned earlier, the problem can be separated into two parts. The properties of the solution have to be understood in the sense that the solution can have a jump at $x_0$
with finite one-sided limits.
\end{remark}

\begin{proof}
The proof is a rather simple application of the Schauder fixed point theorem. We start with the observation that for $0\le y\le x\le x_0$ as well as for $x_0\le x\le y$, we have 
that $G_\eps(x,y)$ is continuously differentiable and satisfies $0<G_\eps(x,y)\le 1$. Note that one of these two cases for the arguments of $G_\eps$ always holds in 
\eqref{eq:g_formula_reg}. 

The property $j/v_\eps\ge 0$ is obviously preserved by $F_\eps$, and so is $M_1[j/v_\eps] = \mathcal{M}_1$, as derived above. The inequality \eqref{jeps-est} for $F[j]$
is a consequence of the above bound for $G_\eps$ and of
$$
    \int_0^\infty h[f]\dd x = \frac{1}{2} M_1[f]^2 \,.
$$
Since obviously continuity on $[0,x_0]$ and on $[x_0,1/\eps]$ is also preserved, we have $F_\eps: B\to B$ with the bounded, closed, convex subset
$$
   B := \{j\in C([0,x_0])\times C([x_0,1/\eps]):\, j/v_\eps \ge 0,\, M_1[j/v_\eps]=\mathcal{M}_1,\, |j| \le J\}
$$
    of the Banach space $C([0,x_0])\times C([x_0,1/\eps])$.

From the differential equation in \eqref{eq_1-reg} we obtain the bound
$$
    |F_\eps[j]'(x)| \le \frac{\mathcal{M}_1 J}{\eps^2} + |h[j/v_\eps](x)| \le \frac{3\mathcal{M}_1 J}{2\eps^2} \,,
$$
implying pre-compactness of $F_\eps[B]$ by the Arzela-Ascoli theorem.

It remains to show the continuity of $F_\eps$ with respect to the supremum norm. This is obvious, however, since $F_\eps[j]$ depends on $j$ in a smooth way, and by the 
regularisation all the coefficient functions as well as the $x$-interval are bounded. Therefore the Schauder theorem can be applied and the proof is complete.
\end{proof}

An existence result for \eqref{eq_1s}, \eqref{BC-infty} is achieved by passing to the limit $\eps\to 0$:

\begin{theo}
Let $v$ satisfy \eqref{eq_3s} and \eqref{v-ass}. Then \eqref{eq_1s}, \eqref{BC-infty} has a weak solution.
\end{theo}

\begin{proof}
Let $j_\eps$ be a solution of \eqref{eq_1-reg} as in Lemma \ref{lem:ex-reg}. Then $f_\eps := j_\eps/v_\eps$ satisfies
\begin{equation}
\begin{array}{ll}
&\partial_x (v_\eps f_\eps) + x\mathcal{M}_1 f_\eps = h\left[f_\eps\right] \,, \label{f-reg} \\ 
& f_\eps(0)=\frac{v(0)f_0}{v_\eps(0)} \,,\qquad f_\eps(1/\eps) = \frac{j_\infty}{v_\eps(1/\eps)}\,,
\end{array}
\end{equation}
and, with $f_\eps(x):=0$ for $x>1/\eps$, also the weak formulation
\begin{eqnarray}
   && \Phi(\infty)j_\infty - \Phi(0)v(0)f_0 -\int_0^\infty \Phi'(x) v_\eps(x)f_\eps(x)\dd x  \nonumber\\
   && = \frac{1}{2} \int_0^\infty \int_0^\infty xf_\eps(x)y f_\eps(y)\Phi(x+y)\dd x\,\dd y -  \mathcal{M}_1\int_0^\infty xf_\eps(x)\Phi(x) \dd x \,,  \label{weak-reg}     
\end{eqnarray}
with $\Phi$ as in Definition \ref{def:weak_sol} and with $\eps$ small enough such that $1/\eps$ is outside the support of $\Phi'$.

The properties of $v$ and of the solution of the regularised problem shown in Lemma \ref{lem:ex-reg} imply
$$
   \int_0^\infty f_\eps\,\dd x \le \frac{x_0}{2} \frac{J}{(x_0-x_0/2)\min_{[0,x_0/2]} w} + \int_{x_0/2}^\infty \frac{x}{x_0/2} f_\eps \,\dd x 
   \le \frac{J}{\min_{[0,x_0/2]} w} + \frac{2\mathcal{M}_1}{x_0} \,,
$$
and, thus, uniform-in-$\eps$ boundedness of $f_\eps$ in $L^1(\R^+)$. Since $M_1[f_\eps]=\mathcal{M}_1$, $\{f_\eps\}_\eps$ is a tight family of measures,
and there exists by the Prokhorov theorem a weak accumulation point $f\in L^1(\R^+)$ with $xf\in L^1(\R^+)$. 

Since $xf_\eps$ appears in the weak formulation, we need a little more to pass to the limit there. With $\delta := \beta/2-1>0$ we have
\begin{eqnarray*}
  \int_0^\infty x^{1+\delta}f_\eps \,\dd x &=& \int_0^{x_0/2} x^{1+\delta}f_\eps \,\dd x + \int_{x_0/2}^{\overline x} x^{1+\delta}f_\eps \,\dd x + \int_{\overline x}^\infty x^{1+\delta}f_\eps \,\dd x \\
  &\le& \frac{J (x_0/2)^{1+\delta}}{\min_{[0,x_0/2]} w} + {\overline x}^\delta \mathcal{M}_1 + \frac{J}{A} \int_{\overline x}^\infty x^{1+\delta-\beta}\dd x \\
  &=& \frac{J (x_0/2)^{\beta/2}}{\min_{[0,x_0/2]} w} + {\overline x}^{\beta/2-1} \mathcal{M}_1 + \frac{J{\overline x}^{1-\beta/2}}{A(\beta/2-1)} \,.
\end{eqnarray*}
This uniform estimate implies that also $\{xf_\eps\}_\eps$ is tight. We apply again the Prokhorov theorem to a sequence $xf_{\eps_n}$, where $f_{\eps_n}\rightharpoonup f$.
For a subsequence, $xf_\eps$ converges weakly, and the limit has to be $xf$. Now we are ready to pass to the limit in \eqref{weak-reg}. Since $v_\eps$ converges strongly,
we have $v_\eps f_\eps\rightharpoonup vf$. Finally, we use that the tensor product of weakly converging measures converges weakly to the tensor product of the limits, which allows
to pass to the limit on the right hand side of \eqref{weak-reg}.
\end{proof}

\subsection{Qualitative properties of steady states}\label{sec:qualitative}

Since the bound \eqref{jeps-est} remains valid in the limit $\eps\to 0$, solutions of \eqref{eq_1s}, \eqref{BC-infty} satisfy
\be\label{f-est}
     0 \le f(x) \le \frac{J}{|v(x)|}  \,,\qquad x\ge 0\,. 
\ee
Therefore (by \eqref{eq_3s}) $f$ has at most one singularity at $x=x_0$. However, close to $x_0$ the estimate is too pessimistic, since we know that $f$ is integrable.
On the other hand, the estimate is sharp (up to the constant) concerning the  as $x\to\infty$, as long as $j_\infty<0$ in \eqref{BC-infty}. For $j_\infty=0$ we refer
to the previous section.

We shall get more precise information on the behaviour around $x=x_0$ and start with the mild formulation, formally passing to the limit $\eps\to 0$ in \eqref{eq:g_formula_reg}:
\begin{equation*}
  v(x)f(x) = \left\{ \begin{array}{ll} 
  G(x,0)v(0)f_0 + \int_0^x G(x,y)h\left[f\right](y)\dd y \quad & \text{for }0\le x\le x_0\,,\\
  G(x,\infty)j_\infty -\int_x^\infty G(x,y) h\left[f\right](y)\dd y & \text{for }x>x_0\,,
       \end{array}\right.
\end{equation*}
with 
$$
  G(x,y) := \exp\left(-\mathcal{M}_1 \int_y^x \frac{z}{v(z)}\dd z\right) \,.
$$
Note that $G(x,\infty)$ is well defined by \eqref{v-ass}. We need properties of $G$ and of $h[f]$:

\begin{lemma}\label{lem:G}
Let $v$ satisfy \eqref{eq_3s}. Then there exists a strictly positive function $G_0\in C((\R^+)^2)$, such that
$$
  G(x,y) = G_0(x,y)\left|\frac{x-x_0}{y-x_0}\right|^c   \quad\text{for } x,y\ge 0 \,,\qquad\text{with } c = \frac{\mathcal{M}_1 x_0}{w(x_0)} \,.
$$
\end{lemma}
\begin{proof}
The integral in the exponent of $G(x,y)$ can be written as
\begin{eqnarray*}
\int_y^x \frac{z}{v(z)}\dd z &=& \frac{x_0}{w(x_0)} \int_y^x \frac{1}{x_0-z}\dd z 
+ \int_y^x \frac{1}{x_0-z}\left( \frac{z}{w(z)} - \frac{x_0}{w(x_0)}\right)\dd z \\
&=& \frac{x_0}{w(x_0)}\log\left|\frac{y-x_0}{x-x_0}\right| + b(x,y) \,,
\end{eqnarray*}
with a continuous function $b$. This follows from an explicit integration and the smoothness assumption on $w$, and it completes the proof with 
$G_0(x,y) = \exp(-\mathcal{M}_1 b(x,y))$.
\end{proof}

\begin{remark}
A similar treatment for $G(x,\infty)$ shows that it also behaves like $|x-x_0|^c$ close to $x=x_0$.
\end{remark}

By \eqref{f-est} the integrand of $h[f](y)$ has at most two singularities at $x_0$ and at $y-x_0$. As long as these are different, i.e. $y\ne 2x_0$, they
are integrable, since $f$ is. Therefore $h[f](y)$ has at most one (integrable) singularity at $y=2x_0$. So this does not interfere with the singularity of $G(x,y)$
at $y=x_0$.

Collecting our results, we first observe that the terms resulting from the boundaries contribute to the solution $f$ a behaviour like $|x-x_0|^{c-1}$. For the integral
terms, there are 3 cases:\\
a) If $c<1$, then $|y-x_0|^{-c}$ is integrable at $y=x_0$, and there results another contribution like $|x-x_0|^{c-1}$.\\
b) If $c=1$, then we obtain a contribution like $-\log|x-x_0|$.\\
c) If $c>1$, the contribution from the integral does not have a singularity.\\
Obviously these contributions are always the dominant ones. We collect our results:

\paragraph{behaviour of solutions of \eqref{eq_1s}, \eqref{BC-infty} close to $x=x_0$:} The behaviour depends on the value of $c= \frac{\mathcal{M}_1 x_0}{w(x_0)}$.
\begin{itemize}
\item If $c<1$, $f$ has a singularity with $f(x) \approx C|x-x_0|^{c-1}$ as $x\to x_0$.
\item If $c=1$, $f$ still has a singularity with $f(x) \approx -C\log|x-x_0|$ as $x\to x_0$.
\item If $c>1$, $f$ is bounded.
\end{itemize}

\section{Numerical results}\label{sec:num}

As an illustration of some of our analytical results, numerical simulations of the time dependent problem have been carried out. The discretization of \eqref{eq:main}
with respect to the aggregate size $x$ involves restriction to a finite maximal size $X>x_0$ (assuming \eqref{eq_3s}), chosen together with the grid spacing $\Delta x$ such that for 
$$
   x_j := j\Delta x \,,\quad j=0,\ldots, J\,, 
$$
we have $x_J = X$ and $x_{j_0} = x_0$.
With the time discretization
$$
   t_k := k\Delta t \,,\quad k\ge 0 \,,
$$
we denote the approximation of $f( t_k,x_j)$ by $f_j^k$. We choose an explicit time
discretization with a conservative discretization of the transport:
\begin{align*}
\frac{{f}_{j}^{k+1} - {f}_{j}^k}{\Delta t} + \frac{F_{j+1/2}^k - F_{j-1/2}^k}{\Delta x} &= \frac{\Delta x}{2} \sum_{i=0}^j K_{j-i,i} f_{j-i}^k f_i^k - \Delta x \sum_{i=0}^{J-j} K_{j,i} f_{j}^k f_i^k =: Q_j^k 
\end{align*}
with the upwinding flux  
\begin{align*}
    F_{j+1/2}^k &= \begin{cases} 
        v(x_j) f_j^k & \text{for } j < j_0\,, \\
        v(x_{j+1}) f_{j+1}^k& \text{for } j \ge j_0\,,
    \end{cases}       
\end{align*}
and with the boundary conditions
\begin{equation}
    f_0^k = f_0 \,, \quad f_J^k = 0\,, \qquad k>0 \,,
\end{equation}
where the latter approximates \eqref{BC-inf} with $j_\infty = 0$. We observe that the restriction $i \le J-j$ in the loss term excludes coagulation events producing aggregates of size larger than $X$.

For stability, the time step has to be chosen small enough such that the CFL condition 
\begin{equation}
    \frac{\Delta x}{\Delta t} \geq \max_{0\le x\le X} |v(x)|
\end{equation}
is satisfied. The discretization is structure preserving in the sense that the discretization of the transport conserves the number of aggregates (locally) by being conservative, and the discretization of the coagulation conserves the discrete total mass
\begin{equation}
    \Delta x\sum_{j=0}^J x_j f_j^k \,,
\end{equation}
since
\begin{equation}
    \begin{split}
    \sum_{j=0}^J x_j Q_j^k &= \frac{\Delta x}{2} \sum_{j=0}^J x_j \sum_{i=0}^j K_{j-i,i} f_{j-i}^k f_i^k - \Delta x \sum_{j=0}^J x_j \sum_{i=0}^{J-j} K_{j,i} f_{j}^k f_i^k\\ 
    &=  \frac{\Delta x}{2} \sum_{j=0}^J f_j^k \sum_{i=0}^{J-j} (x_i + x_j) K_{i,j} f_i^k\, - \Delta x \sum_{j=0}^J x_j \sum_{i=0}^{J-j} K_{j,i} f_{j}^k f_i^k = 0\,,
    \end{split}
\end{equation}
where in the last equality we have used the symmetry of $K_{i,j}$.

In the following, simulation results will be presented for $K(x,y)=xy$,
$f_0=1$, and various choices of the transport velocity $v(x)$. For the discretization, $X=20$ has been used as well as various choices of $\Delta x$.

\paragraph{Comparison with an explicitly known steady state:} As shown in Section \ref{sec:inverse_problem}, the choice
$$
  v(x) = \frac{1}{12}\left(6+6x-3x^2-x^3\right)
$$
corresponds to the steady state $f_\infty(x) = e^{-x}$. Figure \ref{fig:num_sol_multiplicativekernel} shows results of a simulation with this transport velocity and with initial data $f(x,0) = e^{- (x - 6)^2}$. The numerical steady state is close to $f_\infty$, except for a kink at the zero $x_0$ of the transport velocity. This is a numerical artefact, which is understandable because of the singularity at $x=x_0$ of the steady state equation. As an indication of the convergence of the numerical method, the discretization error at the steady state has been computed for five different grids. The influence of the above mentioned kink has been diminished by using a discrete $L^1$-norm for measuring the error:

\begin{center}
\begin{tabular}{|c|c|}
\hline
$\Delta x$ & $L^1$-error \\ \hline
$1.0\mathrm{e}^{-1}$ & 0.035 \\
$6.0\mathrm{e}^{-2}$ & 0.021 \\
$3.6\mathrm{e}^{-2}$ & 0.013 \\
$2.2\mathrm{e}^{-2}$ & 0.008 \\
$1.3\mathrm{e}^{-2}$ & 0.0049 \\ \hline
\end{tabular}
\end{center}
\begin{figure}
    \centering
    \includegraphics[width=0.85\linewidth]{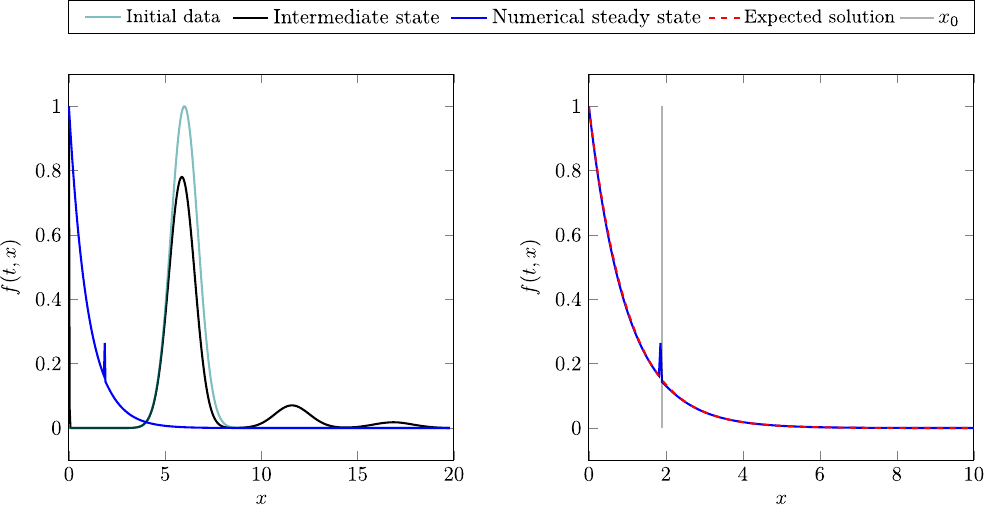}
    \caption{{\bf Snapshots of the numerical solution for the case of an explicitly known steady state.} On the left are depicted the initial data (green), an intermediate time (black), and the numerical steady state (blue). The intermediate state shows that the initial peak is moved to the left by transport, and additional peaks are created by coagulation. The kink at $x=x_0$ of the numerical steady state is a numerical artefact due to the singularity of the steady state equation. On the right, a comparison between the numerical steady state (blue) and the exact steady state (red dashed) is displayed.}
    \label{fig:num_sol_multiplicativekernel}
\end{figure}
%    \begin{figure}[t]
%    \hspace{1cm}
%        \centering
        % Second figure with subfigures
%        \begin{subfigure}{0.4\textwidth}
%            \centering
%            \includegraphics[width=\textwidth]{error_multiplicative.pdf}
%        \end{subfigure}
%        \hfill
%        \begin{subfigure}{0.4\textwidth}
%            \centering
%            \includegraphics[width=\textwidth]{error_multiplicativelog.pdf}
%        \end{subfigure}
%        \hspace{1.3cm}
%        \caption{\textbf{Time evolution of the $\boldsymbol{L^1}$-discrete error norm for $\boldsymbol{K(x,y)=xy}$. (a)} Discrete \( L^1 \)-error in linear scale. The error decreases with respect to smaller values of $\Delta x$. The results are showed for $\Delta x = 0.12$ (blue dashed line), $\Delta x = 0.06$ (green dashed line) and $\Delta x = 0.03$ (gray dashed line). \textbf{(b)} Discrete \( L^1 \)-error in semi-log scale. As in the case of the linear coagulation kernel, the error decreases over time, but it does not follow a clear decay trend.}
%        \label{fig:error_multiplicative}
%    \end{figure}

\paragraph{The behaviour close to \texorpdfstring{$\boldsymbol{x_0}$}{TEXT}:}
In Section~\ref{sec:qualitative}, an analytical investigation of the behaviour of steady states around the singularity at $x=x_0$ has been carried out. This is continued here by numerical simulations, considering different cases of the value of the critical parameter $c$ 
(see Section \ref{sec:qualitative} for its definition).
Since the speed $|v|$ increases rapidly with respect to $x$, we restrict the spatial domain to $X = 7$. This allows the code to run in a reasonable time while preventing the CFL condition from imposing an excessively restrictive constraint on the time step $\Delta t$.
\begin{figure}[htbp]
    \centering
    \vspace{-0.5cm}
    \hspace{1em}
    \begin{subfigure}[b]{0.45\textwidth}
        \includegraphics[width=\textwidth]{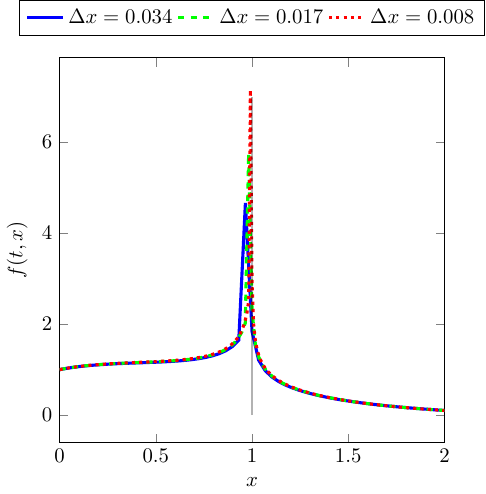}
        \caption{}
    \end{subfigure}
    \hfill
    \begin{subfigure}[b]{0.465\textwidth}
        \includegraphics[width=\textwidth]{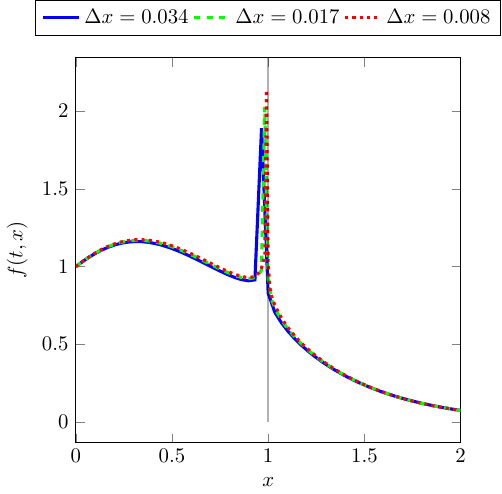}
        \caption{}
    \end{subfigure}
    \hspace{1em}
    
    \vspace{0.5em}
    \hspace{1em}
    \begin{subfigure}[b]{0.45\textwidth}
        \includegraphics[width=\textwidth]{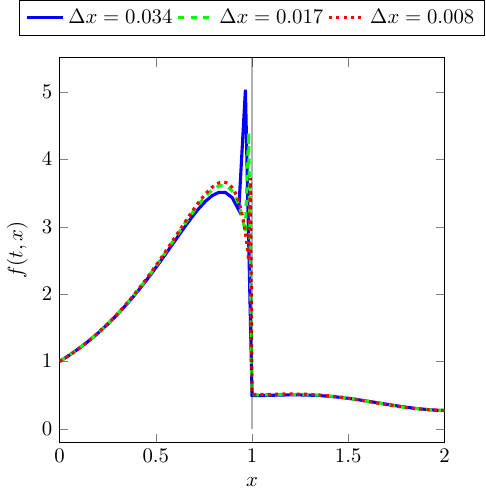}
        \caption{}
    \end{subfigure}
    \hfill
    \begin{subfigure}[b]{0.405\textwidth}
        \includegraphics[width=\textwidth]{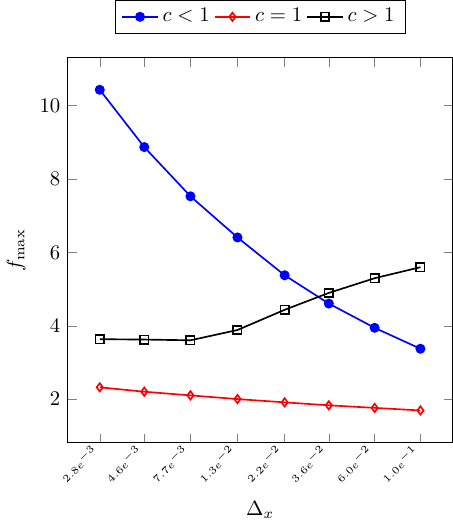}
        \caption{}
    \vspace{-0.3em}
    \end{subfigure}
    \hspace{2.5em}
    
\caption{\textbf{behaviour of numerical steady states near \( \boldsymbol{x_0}\).} Numerical steady states for different values of $c$, computed on grids with \( \Delta x = 0.034 \) (blue solid line), \( \Delta x = 0.017 \) (green dashed line), and \( \Delta x = 0.008 \) (red dotted line). \textbf{(a)} \( c < 1 \),  
\textbf{(b)} \( c = 1 \), \textbf{(c)} \( c > 1 \). \textbf{(d)} Peak values of the numerical steady states for different values of \( \Delta x \), for \( c < 1 \) (blue), \( c = 1 \) (red), and \( c > 1 \) (black). }

    \label{fig:qualitative}
\end{figure}

\noindent\textbf{Case \texorpdfstring{$\boldsymbol{c<1}$}
{TEXT}:} Choosing \( v(x) = (1 - x)(1 + x^{3/2}) \), one easily computes \( c = \frac{1}{\sqrt{2}} < 1 \), hence the steady state solution is expected to behave as \( C |x - x_0|^{c - 1} \) near \( x = x_0 \). Consistent with this analytical result, the numerical solution approximates a singularity as \( x \to x_0 \), as shown in Figure~\ref{fig:qualitative}a. It displays the numerical steady state for decreasing values of \( \Delta x \), showing the formation of a sharp peak near \( x_0 \). The maximum value of this peak increases as the $\Delta x$ decreases, as also shown in Figure~\ref{fig:qualitative}d, consistent with the presence of a local singularity in $x_0$. A video with the time evolution of the numerical solution is available \href{https://doi.org/10.6084/m9.figshare.29590550}{here}.

\vspace{0.2cm}

\noindent\textbf{Case \texorpdfstring{$\boldsymbol{c=1}$}
{TEXT}:} Choosing \( v(x) = (1-x)(1+x^2)/2 \), it is easily checked that \( c = 1 \), hence the steady state solution is expected to behave as \(  -C\log|x-x_0| \) near \( x = x_0 \). As in the previous case, the behaviour of the numerical solution is consistent with the analytical results (see Figure~\ref{fig:qualitative}b), with a slower
growth of the peak (see Figure~\ref{fig:qualitative}d) as expected.
A video with the time evolution of the numerical solution is available \href{https://doi.org/10.6084/m9.figshare.29590508}{here}.

\vspace{0.2cm}

\noindent\textbf{Case \texorpdfstring{$\boldsymbol{c>1}$}
{TEXT}:} With \( v(x) = (1-x)(4-3x+x^2) \) we have \( c =\sqrt{2} >1 \), hence the steady state solution is expected to be bounded near \( x = x_0 \). Although the numerical solution  (see Figure~\ref{fig:qualitative}c) exhibits a peak for coarser grids,
we seem to observe convergence to a jump discontinuity at $x=x_0$ upon refinement. A video with the time evolution of the numerical solution is available \href{https://doi.org/10.6084/m9.figshare.29591678}{here}.

 \section*{Conclusion}

We were originally motivated to develop a qualitative model of autophagy~\cite{delacour2022mathematical}, a phenomenon in which the size of aggregates (called autophagosomes) is regulated (see~\cite{jin2014regulation}). We  introduced a variant of the growth-coagulation equation in which polymerisation dominates for small sizes and depolymerisation dominates for large sizes. We obtained the existence of steady states for this equation with a multiplicative coagulation kernel as soon as depolymerisation is strong enough to counterbalance the well-known gelation effect of this kernel.  This is a mirror effect  of the conclusion in~\cite{calvo2018long}, where a steady state is obtained for a non linear growth-fragmentation equation as soon as growth dominates at large sizes, in order to counterbalance fragmentation, and decay dominates at small sizes to feed larger aggregates.  We have also investigated the relationship between the decay rate and the steady profile for large sizes. Our numerical simulations suggest that the solution converges to the steady state obtained theoretically; the proof of this convergence remains to be established.

\printbibliography

\end{document}